\documentclass[reqno,a4paper,12pt]{amsart} 

\usepackage{amsmath,amscd,amsfonts,amssymb}
\usepackage{mathrsfs,dsfont}

\usepackage{tikz}
\usetikzlibrary{patterns}
\usepackage{float}
\usepackage{appendix}
\usepackage{caption}
\usepackage{subcaption}
\usepackage{enumerate}

\numberwithin{equation}{section}
\numberwithin{figure}{section}

\def\R{\mathbb{R}}
\def\C{\mathbb{C}}
\def\Q{\mathbb{Q}}
\def\Z{\mathbb{Z}}

\def\1{\mathds{1}}

\renewcommand\leq{\leqslant}
\renewcommand\geq{\geqslant}

\renewcommand\hat{\widehat}

\newcommand{\ft}[1]{\widehat #1}
\newcommand{\dotprod}[2]{\langle #1 , #2 \rangle}

\theoremstyle{plain}
\newtheorem{thm}{Theorem}[section]
\newtheorem{lem}[thm]{Lemma}
\newtheorem{corollary}[thm]{Corollary}

\newtheorem{problem}[thm]{Problem}

\newtheorem{claim}{Claim}
\newtheorem*{claim*}{Claim}
\newtheorem*{thm*}{Theorem}

\theoremstyle{definition}

\newtheorem*{definition*}{Definition}
\newtheorem*{remarks*}{Remarks}
\newtheorem*{remark*}{Remark}

\newenvironment{enumerate-math}
{\begin{enumerate}
\addtolength{\itemsep}{5pt}
}
{\end{enumerate}}

\newenvironment{enumerate-text}
{\begin{enumerate}
\addtolength{\itemsep}{5pt}
}
{\end{enumerate}}

\begin{document}

\title{Periodic structure of translational multi-tilings in the plane}

\author{Bochen Liu}
\address{National Center for Theoretical Sciences, No. 1 Sec. 4 Roosevelt Rd., National Taiwan University, Taipei, 106, Taiwan}
\email{Bochen.Liu1989@gmail.com}

\thanks{Research supported by ERC Starting Grant No.\ 713927.}
\date{}


\begin{abstract}
	Suppose $f\in L^1(\mathbb{R}^d)$, $\Lambda\subset\mathbb{R}^d$ is a finite union of translated lattices such that $f+\Lambda$ tiles with a weight. We prove that there exists a lattice $L\subset{\mathbb{R}}^d$ such that $f+L$ also tiles, with a possibly different weight. As a corollary, together with a result of Kolountzakis, it implies that any convex polygon that multi-tiles the plane by translations admits a lattice multi-tiling, of a possibly different multiplicity.

	Our second result is a new characterization of convex polygons that multi-tile the plane by translations. It also provides a very efficient criteria to determine whether a convex polygon admits translational multi-tilings. As an application, one can easily construct symmetric $(2m)$-gons, for any $m\geq 4$, that do not multi-tile by translations.

	Finally, we prove a convex polygon which is not a parallelogram only admits periodic multiple tilings, if any.
\end{abstract}
\maketitle

\section{Introduction}
\subsection{Tiling and multiple tiling}
Let $P\subset{\mathbb{R}}^d$ be a convex body and $\Lambda\subset{\mathbb{R}}^d$ be a discrete multi-set, which means $\Lambda$ is discrete and each point has finite multiplicity in $\Z_+$. Denote $\chi_P$ as the indicator function of $P$ and 
$$\delta_\Lambda = \sum_{\lambda\in\Lambda} \delta_\lambda,$$
where $\delta_\lambda$ is the Dirac measure at $\lambda$. We say that $P+\Lambda$ tiles if for almost all $x\in\R^d$,
\begin{equation}
	\label{tiling}
	\chi_P*\delta_\Lambda(x)=\sum_{\lambda\in\Lambda} \chi_P(x-\lambda)=\sum_{\lambda\in\Lambda} \chi_{P+\lambda}(x) = 1.
\end{equation}
We say $P+\Lambda$ multi-tiles, or is a multiple tiling, of multiplicity $k\in\Z_+$, if for almost all $x\in\R^d$,
\begin{equation}
	\label{multi-tiling}
	\chi_P*\delta_\Lambda(x) = k.
\end{equation}
More generally we say $f+\Lambda$ tiles with a weight $w\in\R$, where $f\in L^1(\R^d)$, if for almost all $x\in\R^d$,
$$f*\delta_\Lambda(x) = w.$$
One can see under these definitions $P+\Lambda$ is equivalent to $\chi_P+\Lambda$.

Throughout this paper a lattice of $\R^d$ is a discrete subgroup of the additive group $\R^d$ which is isomorphic to the additive group $\Z^d$.

The study of translational tilings by convex bodies has a long history. It has been known for a long time that the only convex domains that tile the plane by translations are parallelograms and hexagons. In 1885, Fedorov classified three-dimensional convex polytopes which can tile by translations into 5 different combinatorial types. In 1897, Minkowski (\cite{Min97}) showed that if a convex body $P$ tiles $R^d$ by lattice translations, then $P$ must be a centrally symmetric polytope, with centrally symmetric facets. Finally Venkov (\cite{Ven54}) gave a characterization, which was later rediscovered by McMullen (\cite{McM80}), of convex bodies that tile $\R^d$ by translations.
\begin{thm}
	[Venkov, 1954 \& McMullen, 1980]
	Let $P$ be a convex body in $\R^d$. Then $P$ tiles by translations if and only if the following four conditions are satisfied:
\begin{enumerate}
	\item $P$ is a polytope. 
	\item $P$ is centrally symmetric.
	\item All facets of $P$ are centrally symmetric.
	\item Each ``belt" of $P$ consists of $4$ or $6$ facets.
\end{enumerate}
Here by a facet one means a $(d-1)$-dimensional face, and by a belt one means the collection of its facets which contain a translate of a given subfacet, that is, a $(d-2)$-dimensional face, of $P$.
\end{thm}
As a consequence of Venkov-McMullen theorem, it follows that if a convex polytope $P$ tiles, it admits a face-to-face tiling by translates along a certain lattice.
\vskip.125in
The study of multiple translational tilings dates back to 1936, when the famous Minkowski conjecture for tilings was extended to multiple tilings by Furtw\"angler (\cite{Fur36}). For more information about this problem, one can see, for example, \cite{Zon06}, Chapter $6, 7, 8$. It was showed by Bolle (\cite{Bol94}) that in the plane, every convex domain that admits lattice multi-tilings has to be a centrally symmetric polygon. More generally, it is well known 
that a convex body in $\R^d$ that multi-tiles by translations must be a convex polytope (see Appendix). In \cite{GRS12}, Gravin, Robins and Shiryaev showed that these polytopes must be centrally symmetric with centrally symmetric facets. This implies in dimension $2,3$, a convex body $P$ multi-tiles only if it is a zonotope. Therefore in the rest of this paper we assume $P\subset\R^2$ is the zonotope generated by pairwise non-colinear vectors $e_1,\dots, e_m$, of increasing arguments (see Figure \ref{fig:polygon}), that is,
$$P=\left\{\alpha_1 e_1+\cdots+\alpha_m e_m: \alpha_j\in [-\frac{1}{2}, \frac{1}{2}]\right\}. $$
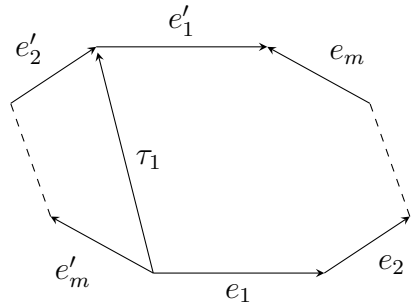
\begin{figure}[H]
\centering
\begin{tikzpicture}[scale=0.75, p2/.style={line width=0.275mm, black}, p3/.style={line width=0.15mm, black!50!white}]

\draw[-stealth] (-1, -2) -- (2, -2);
\draw[-stealth] (2,-2) -- (3.5,-1); 
\draw[-stealth] (2.8, 1) -- (1, 2);
\draw[dashed] (3.5,-1) -- (2.8,1);
\draw[-stealth] (-2, 2) -- (1,2);
\draw[-stealth] (-3.5, 1) -- (-2,2); 
\draw[-stealth]  (-1, -2) -- (-2.8, -1);
\draw[dashed] (-2.8, -1) -- (-3.5,1);

\draw[-stealth] (-1,-2) -- (-1.975,1.9);

\draw (1.9, 1.5) node[anchor=south west]{$e_m$};
\draw (-1.9, -1.5) node[anchor=north east]{$e'_m$};
\draw (2.75, -1.5) node[anchor=north west]{$e_2$};
\draw (-2.75, 1.5) node[anchor=south east]{$e'_2$};
\draw (0.5, -2) node[anchor=north]{$e_1$};
\draw (-0.5, 2) node[anchor=south]{$e'_1$};
\draw (-1.5, 0) node[anchor=west]{$\tau_1$};


\end{tikzpicture}
\caption{The zonotope generated by $e_1, \dots, e_m$, of increasing arguments}
\label{fig:polygon}
\end{figure}

Now edges of $P$ consist of translated $e_1,\dots, e_m$ and their parallel edges $e'_1,\dots, e'_m$. We denote $\tau_j$ as the vector that translates $e_j$ to its parallel edge. 

Besides the structure of $P$, one can also study the structure of the discrete multi-set $\Lambda$ such that $P+\Lambda$ multi-tiles. In 2000, Kolountzakis (\cite{Kol00}) proved the following result in the plane. 
\begin{thm}[Kolountzakis, 2000]\label{Kol00}
Suppose $P\subset\R^2$ is a convex polygon which is not a parallelogram, $\Lambda$ is a multi-set such that $P+\Lambda$ multi-tiles. Then $\Lambda$ must be a finite union of translated lattices.	
\end{thm}

A three dimensional version of this result was later obtained by Gravin, Kolountzakis, Robins and Shiryaev (\cite{GKRS13}). They proved that the same conclusion on $\Lambda$ holds if a convex polytope $P\subset\R^3$ multi-tiles with $\Lambda$ and $P$ is not a two-flat zonotope. They also constructed multiple tilings by two-flat zonotopes where the discrete sets are not finite unions of translated lattices. Here by a two-flat zonotope one means the Minkowski sum of finitely many line segments which lie in the union of two different two-dimensional subspaces. 

If, in particular, $\Lambda$ is given as a lattice, Bolle (\cite{Bol94}) used combinatorial methods to give a characterization of convex polygons that multi-tiles the plane with $\Lambda$. An equivalent formulation is the following.
\begin{thm}
	[Bolle, 1994]\label{Bol94}
Let  $P$ be a convex polygon  in $\R^2$, and $L$ be a lattice in $\R^2$. 
Then $P+L$ multi-tiles if and only if 
$P$ is centrally symmetric, and for each pair of parallel edges $e$ and $e'$ of $P$ one of the following conditions is satisfied:
\begin{enumerate}
\item
	The translation vector, $\tau$, which carries $e$ onto $e'$, is in $L$, or
	\item
	$e\in L$ and there exists $t\in\R$ such that $te+\tau\in L$. 
\end{enumerate}
\end{thm}
A general dimensional version of Bolle's theorem was recently obtained by Lev and the author (\cite{LL18}), via a Fourier-analytic approach.

There is also some work on planar translational multi-tilings with given multiplicities. See, for example, Section 5, 6, 7 in Zong's survey \cite{Zon18}.

\vskip.125in
\subsection{Main results}
Throughout this paper we say a (multiple) tiling $P+\Lambda$ is periodic if there exists a lattice $L$ such that $\Lambda+L=\Lambda$, which is equivalent to that $\Lambda$ is a finite union of translations of $L$.
\subsubsection{Periodic tiling conjecture}
One famous open problem on tiling is the periodic tiling conjecture (\cite{GS87}, \cite{LW96}), which states any region that tiles $\R^d$ by translations has a periodic tiling. Here a region is a closed subset of $\R^d$ whose boundary has measure $0$. 

In the real line this conjecture was confirmed by Lagarias and Wang (\cite{LW96}) for bounded regions. In the plane it is proved, first by Girault-Beauquier and Nivat (\cite{BN91}) with boundary conditions and finally by Kenyon (\cite{Ken92}), for closed topological discs.

For convex bodies (equivalently, convex polytopes) in general $\R^d$, it follows from Venkov-McMullen's result that if $P$ tiles by translations then it admits a lattice tiling. 

The periodic tiling conjecture in $\Z^2$ has been solved by Bhattacharya (\cite{Bha16}). 

It is natural to extend this problem to multiple tilings. The following question was raised by Gravin, Robins and Shiryaev (see Problem 7.3 in \cite{GRS12}).

\begin{problem}\label{GRS12}
Prove or disprove that if a convex polytope multi-tiles $\R^d$ by translations, then it also multi-tiles $\R^d$ by a lattice, for a possibly different multiplicity.
\end{problem}
Problem \ref{GRS12} is open in all dimensions. It is mentioned in \cite{Cha15} that Shiryaev has a proof in the plane, but it is never posted.  Theorem \ref{Kol00} and the discussion afterwards show that, in low dimensions it is quite often that $\Lambda$ is a finite union of translated lattices. Inspired by this, Chan asked the following weaker version in \cite{Cha15}, where he also considered two special cases. 

\begin{problem}
	\label{Cha15}
	Let $P$ be a convex polytope that multi-tiles $\R^d$ with a discrete multiset $\Lambda$, which is a finite union of translated lattices. Prove or disprove that $P$ could multi-tile $\R^d$ with a lattice.
\end{problem}

In this paper we solve Problem \ref{Cha15} in all dimensions. In fact we prove a stronger statement that works for any $f\in L^1(\R^d)$. Fourier analysis and number theory play important roles in the proof.

\begin{thm}\label{main1}
	Suppose $f\in L^1(\R^d)$ and $\Lambda$ is a discrete multi-set such that 
	$$\delta_\Lambda=\sum_{j=1}^n \delta_{L_j+{z}_j},  $$
where $L_j$ are lattices and ${z}_j\in\R^d$. If $f+\Lambda$ tiles with a weight, then for each $j$ there exists a lattice $\widetilde{L_j}$ containing $L_j$ such that $f+\widetilde{L_j}$ also tiles, with a possibly  different weight.
\end{thm}
The author does not know if the assumption that $\Lambda$ is a finite union of translated lattices is necessary. An interesting question is, does there exist a function $f\in L^1(\R^d)$ that tiles by translations but does not tile by any lattice with any weight? This question is also asked by Kolountzakis and Lev in \cite{KL16}, where they construct non-periodic tilings by some $f\in L^1(\R)$, which also admit periodic tilings.

Together with Theorem \ref{Kol00} we solve Problem \ref{GRS12} in the plane.

\begin{corollary}\label{SolProb1}
	If a convex polygon $P$ multi-tiles the plane by translations, it also multi-tiles the plane by a lattice.
\end{corollary}

As a remark, we remind the reader that Theorem \ref{Kol00} remains valid for non-convex polygons with the pairing property (see \cite{Kol00}), which means for each edge $e$ there is precisely one other edge parallel to $e$. Therefore Corollary \ref{SolProb1} holds for non-convex polygons with the pairing property as well.

Independently, Corollary \ref{SolProb1} is obtained by Yang (\cite{Yan18}). Her argument is also based on Theorem \ref{Kol00} of Kolountzakis, while elementary and purely combinatorial, thus completely different from ours. Her method does not seem to yield Theorem \ref{main1}.

\subsubsection{A new characterization of convex polygons that multi-tile the plane}
Since Corollary \ref{SolProb1} holds, Bolle's theorem (Theorem \ref{Bol94}) has automatically become a characterization of convex polygons that multi-tiles. However, it is not a very efficient criteria to determine whether a convex polygon multi-tiles. To apply Bolle's theorem, one needs to find a subset $J$ of $\{1,\dots,m\}$ such that
$$span_\Z\{e_j, \tau_{j'}: j\in J, j'\notin J\}$$
is a lattice. If we check this condition in the brute-force way, the computation complexity is exponential in terms of $m$. Then it is natural to look for a more efficient criteria, where the complexity has polynomial growth as $m$ increases. Our second main result in this paper is a new characterization, as well as an efficient criteria, on convex polygons that multi-tile by translations. The proof is based on Theorem \ref{main1} and Bolle's theorem (Theorem \ref{Bol94}).

\begin{thm}\label{refined}
	Suppose $P$ is a convex symmetric polygon as in Figure \ref{fig:polygon} which is not a parallelogram. Then $P$ admits multiple translational tilings if and only if

\begin{enumerate}
	\item $m$ is odd and $\Lambda_\tau:=span_{\Z}\{\tau_1,\dots,\tau_m\}$ is a lattice, or
	\item $m$ is even and there exists $1\leq j_0\leq m$ such that
	\begin{enumerate}
	\item $\Lambda_{j_0}:=span_{\Z}\{\tau_1,\dots,\tau_{j_0-1}, \tau_{j_0+1},\dots,\tau_m\}$ is a lattice, and
	\item $\det(e_{j_0}, \tau_{j_0})$ is an rational multiple of $\det(\Lambda_{{j_0}})$. 
	\end{enumerate}
\end{enumerate}
Moreover, if $P$ multi-tiles by translations, then \begin{equation}\label{L_P}L_P:=\begin{cases}
	\Lambda_\tau, & m \text{ is odd}\\\bigcap\limits_{j:\, \Lambda_j \text{ is a lattice}} \Lambda_{j}, &m \text{ is even}
\end{cases}\end{equation}
is a lattice and $L_P\cap L$ is a lattice for any lattice multi-tiling $P+L$.
\end{thm}

As an application, one can easily construct convex symmetric $(2m)$-gons, for any $m\geq 4$, that do not multi-tile by translations (see Example \ref{do-not-multi-tile} in Section \ref{example}). As far as the author knows, these are first known symmetric polygons that do not multi-tile by translations. On the other hand, there are many symmetric $(2m)$-gons that do multi-tile by translations. For example $P+L$ multi-tiles if $P$ is symmetric whose vertices lie in a lattice $L$ (see \cite{GRS12}, \cite{LL18}). This means, unlike tiling (Venkov-McMullen), one can not determine whether a polygon multi-tiles only by its combinatorial type.

\subsubsection{Periodic multiple tilings}\label{Periodic}
The last problem we consider is whether a multiple tiling must be periodic. 

In $\R$, it was proved by Lagarias and Wang (\cite{LW96}) that a bounded region only admits periodic tilings. This result was extended by Kolountzakis and Lagarias (\cite{KL96}) (and proved earlier by Leptin and M\"uller in \cite{LM91}) to tilings by a function $f\in L^1(\R)$ with compact support. More precisely they showed if $f+\Lambda$ tiles and $\Lambda$ has bounded density, then $f+\Lambda$ is a finite union of periodic tilings, with weights. Later Kolountzakis and Lev (\cite{KL16}) showed the assumption $f$ has compact support is necessary. They also proved that if the translation set has finite local complexity, then it must be periodic, even if the support of $f$ is unbounded.


In this paper we answer this question for multiple tilings in the plane. The proof starts from Theorem \ref{Kol00} and eventually we improve it from ``a finite union of translated (possibly different) lattices" to ``a finite union of translations of a single lattice".
\begin{thm}\label{periodic-multi-tilings}
	Suppose $P\subset\R^2$ is a convex polygon which is not a parallelogram. Then every multiple tiling of $P$ is periodic.
\end{thm}

It can be seen that in Theorem \ref{periodic-multi-tilings} both non-parallelogram and convexity are necessary. In fact in either case there are multiple tilings $P+\Lambda$ where $\Lambda+\alpha\neq\Lambda$ for any $\alpha\in\R^2\backslash\{0\}$. One can see Example \ref{non-periodic} in Section \ref{example}.

In \cite{LW96}, Lagarias and Wang not only proved that $\Omega+\Lambda$ tiles the real line implies
$$\Lambda=\alpha\Z+\{\beta_1,\dots,\beta_n\},  $$
but also showed $\beta_i-\beta_j, \forall 1\leq i, j\leq n$, must be a rational multiple of $\alpha$. This rationality result does not hold for tilings of compactly supported functions (\cite{KL96}). Also it is easily seen to fail for tilings in higher dimensions (parallelepipeds), or decomposable multi-tilings ($\Lambda\cup(\Lambda+z)$ where $\Omega+\Lambda$ tiles). In this paper, we give examples to show, even for indecomposable multi-tilings by convex symmetric polygons that are not parallelograms, rationality still fails. See Example \ref{non-rationality} in Section \ref{example}.

\subsection{Other applications}
\subsubsection{Dimension $1$}
With results of Lagarias-Wang, Kolountzakis-Lagarias, Leptin-M\"uller, Kolountzakis-Lev introduced above, Theorem \ref{main1} implies the following. The case of $f$ of compat support has been proved in \cite{KL96}, Theorem 1.2.
\begin{corollary}
 Suppose 
 $f\in L^1(\R)$ with compact support and $\Lambda \subset\R$ has bounded density, or $f\in L^1(\R)$ and $\Lambda\subset\R$ has finite local complexity. If $f+\Lambda$ tiles with a weight, then there exists a lattice $L\subset\R$ such that $f+L$ also tiles, with a possibly different weight.
\end{corollary}

\subsubsection{Higher dimensions}
As we introduced right after Theorem \ref{Kol00}, Gravin, Kolountzakis, Robins and Shiryaev (\cite{GKRS13}) proved that if $P\subset\R^3$ is a convex polytope but not a two-flat zonotope and $P+\Lambda$ multi-tiles, then $\Lambda$ must be a finite union of translated lattices. By Theorem \ref{main1} we obtain the following partial result on Problem \ref{GRS12} in $\R^3$.
\begin{corollary}\label{Cor-3D}
	Suppose $P\subset\R^3$ is a convex polytope, which is not a two-flat zonotope, and $P$ multi-tiles by translations. Then there exists a lattice $L\subset\R^3$ such that $P+L$ multi-tiles.
\end{corollary}
Although Gravin, Kolountzakis, Robins and Shiryaev (\cite{GKRS13}) gave examples of two-flat zonotopes which admit weird multiple tilings, their examples admit periodic multi-tilings as well. So whether Corollary \ref{Cor-3D} holds for general convex polytopes in $\R^3$ is still unknown.

There is very little known in dimension $4$ and higher. In fact there exists centrally symmetric polytopes, with centrally symmetric facets, that multi-tile by translations but are not zonotopes (e.g. the $24$-cell in $\R^4$), which makes the study of multiple tilings in higher dimensions more difficult than in lower dimensions.

\subsubsection{Riesz basis}
We say $\Omega\subset\R^d$ admits an exponential Riesz basis if there exists a discrete set $\Lambda\subset\R^d$, $A, B>0$ such that
$$ A||f||^2_{L^2(\Omega)}\leq \sum_{\lambda\in\Lambda} \left|\widehat{f\chi_\Omega}(\lambda)\right|^2 \leq B||f||^2_{L^2(\Omega)}$$
and
$$ A\sum_{\lambda\in\Lambda} |c_\lambda|^2\leq \int_\Omega \left|\sum_{\lambda\in\Lambda} c_\lambda e^{-2\pi i x\cdot\lambda}\right|^2\,dx \leq B\sum_{\lambda\in\Lambda} |c_\lambda|^2. $$
The connection between multiple tiling and exponential Riesz basis was first discovered by Grepstad and Lev (\cite{GL14}) in 2014, and later reproved by Kolountzakis (\cite{Kol15}) in 2015 with an elementary argument. They proved that if a bounded region $\Omega\subset\R^d$ multi-tiles by a lattice, then it admits an exponential Riesz basis.

Since we have proved that every convex polygon $P$ that multi-tiles the plane admits a lattice multi-tiling, it follows that a convex polygon admits an exponential Riesz basis if it multi-tiles (no need to assume lattice multi-tiling). Also a sufficient condition on the existence of exponential Riesz bases follows from Theorem \ref{refined}.
\begin{corollary}
Let $P$ be a convex polygon in the plane. Then $P$ admits an exponential Riesz basis if it admits translational multi-tilings.
\end{corollary}
Similar to the remark right after Corollary \ref{SolProb1}, this corollary also holds for non-convex polygons with the pairing property.

During the referee process, Debernardi and Lev \cite{DL19} proved a much stronger result. They show that any zonotope in $\R^d$, $d\geq 1$, admits an exponential Riesz basis. Notice in the plane any convex multiple tile is necessarily a zonotope, but not vice versa.




{\bf Organization.} This paper is organized as follows. In Section \ref{Prelim} we review useful tools from Fourier analysis and number theory. Then we prove Theorem \ref{main1}, \ref{refined}, \ref{periodic-multi-tilings} in Section \ref{pf-main1}, \ref{pf-refined}, \ref{pf-perodic}, respectively. In Section \ref{example} we discuss some examples. In the Appendix, we give a proof of that any convex body in $\R^d$ that multi-tiles by translations must be a convex polytope.

{\bf Notation.} Throughout this paper a lattice of $\R^d$ is a discrete subgroup of the additive group $\R^d$ which is isomorphic to the additive group $\Z^d$.

We say $\Lambda$ is a finite union of translated lattices if it is a multi-set and
$$\delta_\Lambda=\sum_{j=1}^n \delta_{L_j+{z}_j},  $$
where $L_j$ are (possibly different) lattices and ${z}_j\in\R^d$.

We say a (multiple) tiling $P+\Lambda$ is periodic if there exists a lattice $L$ such that $\Lambda+L=\Lambda$, which is equivalent to that $\Lambda$ is a finite union of translations of $L$.

{\bf Acknowledgment.} The author would like to thank Nir Lev for bringing this subject to attention, useful discussions, and comments on manuscript.

\section{Preliminaries}\label{Prelim}
\subsection{Fourier analysis}
Let $L$ be a lattice in $\R^d$ and denote its dual lattice as
$$L^*=\{\lambda^*\in\R^d: \lambda^*\cdot\lambda\in\Z, \forall \lambda\in L\}.  $$
For $f\in L^1(\R^d)$, define its Fourier transform as
$$\hat{f}(\xi) = \int_{R^d} e^{-2\pi i x\cdot\xi} f(x)\,dx.$$

Denote $\det(L)$ as the volume of a fundamental domain of a lattice $L\subset\R^d$. The well-known Poisson summation formula can be stated as
$$\sum_{\lambda\in L}\phi (\lambda+{z})=\frac{1}{\det(L)}\sum_{\lambda^*\in L^*} e^{2\pi i \lambda^*\cdot {z}} \,\hat{\phi}(\lambda^*) $$
for any Schwartz function $\phi$. In the sense of distributions, it is equivalent to
\begin{equation}\label{Poisson-distribution}\widehat{\delta_{L+{z}}}(\xi)= \frac{1}{\det(L)}\sum_{\lambda^*\in L^*}e^{-2\pi i \xi\cdot{z}}\,\delta_{\lambda^*}(\xi)=\frac{e^{-2\pi i \xi\cdot{z}}}{\det(L)}\,\delta_{L^*}(\xi).\end{equation}

We also need the following well-known lemma that connects Fourier analysis and multiple tilings. 
\begin{lem}\label{vanish}
	Let $f\in L^1(\R^d)$ and $L$ be a lattice in $\R^d$. Then  $f+L$ tiles with a weight if and only if $\hat{f}$ vanishes on $L^*\backslash\{0\}$.
\end{lem}
We give a proof below for completeness.
\begin{proof}
We may assume $L = \Z^d$. Let
\[
F(x) := \sum_{\lambda \in \Z^d} f(x - \lambda),
\]
then $F$ is a $\Z^d$-periodic function whose Fourier series is given by
\[
\sum_{\lambda \in \Z^d}  \hat{f}(\lambda) e^{2 \pi i \dotprod{\lambda}{x}}
\]
(see e.g.\ \cite{SW71}, Chapter VII, Theorem 2.4).
Hence $f$ coincides a.e.\ with a constant function, if and only if $\ft{f}$ vanishes on $\Z^d \setminus \{0\}$.
\end{proof}

\subsection{Solutions of linear equations}
Let $K$ be an algebraically closed field of characteristic $0$ and denote $K\backslash\{0\}$ as its multiplicative group of nonzero elements. Let $(a_1,\dots, a_n)\in (K\backslash\{0\})^n$ and $\Gamma$ be a subgroup of $(K\backslash\{0\})^n$. One may ask how many solutions does the linear equation 
\begin{equation}\label{equation}a_1x_1+\cdots+a_nx_n=1  \end{equation}
have with $(x_1,\dots, x_n)\in\Gamma$. This problem has been studied for a long time in number theory and literature dates back to early 1930s (e.g., \cite{Mah33}). Finally, in 2002, Evertse, Schlickewei and Schmidt (\cite{ESS02}) proved the following celebrated result.

We say $\Gamma$ has finite rank $r$, if there exists a finitely generated subgroup $\Gamma_0$ of $\Gamma$, again of rank $r$, such that the factor group $\Gamma/\Gamma_0$ is a torsion group.

\begin{thm}
	[Evertse, Schlickewei, Schmidt, 2002]
	With notation above, suppose $\Gamma$ has finite rank $r$. Then $A(a_1,\dots,a_n,\Gamma)$, the number of non-degenerate solutions $(x_1,\dots,x_n)\in\Gamma$ of equation \eqref{equation} satisfies the estimate
	$$ A(a_1,\dots,a_n,\Gamma)\leq A(n,r) = \exp\left((6n)^{3n}(r+1)\right). $$
	Here a solution $(x_1,\dots, x_n)\in\Gamma$ is called non-degenerate if $\sum_{i\in I} a_i x_i\neq 0$ for every nonempty subset $I\subset\{1,\dots,n\}$.
\end{thm}

In particular, given $z_1,\dots,z_n\in\R^d$ and a lattice $L\subset\R^d$, take $K=\C$ and $$\Gamma_{z_1,\dots,z_n,L} = \{(e^{-2\pi i \lambda\cdot z_1},\dots,e^{-2\pi i \lambda\cdot z_n}):\lambda\in L\}.$$
The following corollary plays an important role in our proof of Theorem \ref{main1}.
\begin{corollary}\label{Sol-linear equation}
	Given $a_1,\dots,a_n\in\C\backslash\{0\}$, $z_1,\dots,z_n\in \R^d$ and a lattice $L\subset\R^d$, then for any nonempty subset $I\subset\{1,\dots,n\}$, the linear equation 
	$$\sum_{i\in I} a_ix_i=1$$
	has finitely many non-degenerate solutions in $\Gamma_{z_1,\dots,z_n,L}$. In particular,
$$\#\{(x_1,\dots,x_n)\in\Gamma_{z_1,\dots,z_n,L}: \exists\ \emptyset\neq I\subset\{1,\dots,n\}, \sum_{i\in I} a_ix_i=1  \}<\infty  $$
\end{corollary}

\section{Proof of Theorem \ref{main1}}\label{pf-main1}
Now $\Lambda$ is a finite union of translated lattices,
$$\delta_\Lambda = \sum_{j=1}^n \delta_{L_j+{z}_j}. $$
Without loss of generality, we may assume $L_1=\Z^d$, $z_1=0$ and $n\geq 2$. Then, by Lemma \ref{vanish}, it suffices to find a lattice $L^*\subset\Z^d$ such that $\hat{f}$ vanishes on $L^*\backslash\{0\}$.

By Poisson summation formula \eqref{Poisson-distribution}, 
$$\hat{\delta_\Lambda}(\xi) = \delta_{\Z^d}+\sum_{j=2}^n \frac{e^{-2\pi i \xi\cdot {z}_j}}{\det{(L_j)}}\delta_{L_j^*}(\xi). $$
Denote 
\begin{equation}\label{omega}\omega_j(\lambda^*)=\begin{cases}1, \ \lambda^*\in L_j^* \\ 0, \ otherwise\end{cases}.  \end{equation}
In a small neighborhood $U_{\lambda^*}$ of each $\lambda^*\in\Z^d$,
$$\hat{\delta_\Lambda}\big|_{U_{\lambda^*}} = \left(1+\sum_{j\geq 2} \omega_j(\lambda^*)\frac{e^{-2\pi i \lambda^*\cdot {z}_j}}{\det{(L_j)}}\right)\delta_{\lambda^*}. $$

Since $f*\delta_\Lambda$ is a constant almost everywhere, its Fourier transform is a multiple of $\delta_0$. Therefore on a small neighborhood $U_{\lambda^*}$ of each $\lambda^*\in\Z^d\backslash\{0\}$,
$$0=\widehat{f*\delta_\Lambda}\big|_{U_{\lambda^*}}=\hat{f} \cdot\hat{\delta_\Lambda} \big|_{U_{\lambda^*}} = \left(1+\sum_{j\geq 2} \omega_j(\lambda^*)\frac{e^{-2\pi i \lambda^*\cdot {z}_j}}{\det{(L_j)}}\right)\hat{f}(\lambda^*)\,\delta_{\lambda^*},  $$
which implies that for any $\lambda^*\in\Z^d\backslash\{0\}$, either $\hat{f}(\lambda^*)=0$ or
$$\sum_{j\geq 2} -\omega_j(\lambda^*)\frac{1}{\det{(L_j)}}e^{-2\pi i \lambda^*\cdot {z}_j}=1.  $$

Therefore, to find a lattice $L^*\subset\Z^d$ such that $\hat{f}$ vanishes on $L^*\backslash\{0\}$, it suffices to find a lattice $L^*\subset\Z^d$ such that
$$\sum_{j\in I} -\frac{1}{\det{(L_j)}}e^{-2\pi i \lambda^*\cdot {z}_j}\neq 1  $$
for any $\lambda^*\in L^*\backslash\{0\}$ and any nonempty subset $I\subset\{2,\dots,n\}$.


\begin{lem}
$$\left\{\lambda^*\in\Z^d: \exists\, \emptyset\neq I\subset \{2,\dots,n\}, \sum_{j\in I} -\frac{1}{\det{(L_j)}}e^{-2\pi i \lambda^*\cdot {z}_j}= 1\right\}$$
is a finite union of cosets of subgroups of $\Z^d$, where each coset does not contain the origin.
\end{lem}
\begin{proof}
By Corollary \ref{Sol-linear equation},
\begin{equation}\label{finite-sol}\left\{(e^{-2\pi i \lambda^*\cdot z_1},\dots,e^{-2\pi i \lambda^*\cdot z_n}): \lambda^*\in\Z^d, \exists\, \emptyset\neq I\subset \{2,\dots,n\}, \sum_{j\in I} -\frac{e^{-2\pi i \lambda^*\cdot {z}_j}}{\det{(L_j)}}= 1\right\}\end{equation}
has finitely many elements. Then it suffices to show for each solution $(e^{-2\pi i \lambda_0^*\cdot z_1},\dots,e^{-2\pi i \lambda_0^*\cdot z_n})$, $\lambda^*_0\in\Z^d$, of
$$\sum_{j\in I} -\frac{1}{\det{(L_j)}}e^{-2\pi i \lambda^*\cdot z_j}= 1,$$
for some $\emptyset\neq I\subset \{2,\dots,n\}$, the set
$$L_{\lambda_0^*}:=\left\{\lambda^*\in\Z^d: e^{-2\pi i \lambda^*\cdot {z}_j}=e^{-2\pi i \lambda_0^*\cdot z_j}, j\in I\right\}$$
is a coset of a subgroup of $\Z^d$ which does not contain the origin. It is easy to see $L_{\lambda^*_0}$ is a coset of
$$\left\{\lambda^*\in\Z^d: e^{-2\pi i \lambda^*\cdot {z}_j}=1, j\in I\right\}.$$ 
For any $\emptyset\neq I\subset \{2,\dots,n\}$, since $(1,\dots,1)$ is not a solution of 
$$\sum_{j\geq 2,\  j\in I} -\frac{x_j}{\det{(L_j)}}= 1,$$ one concludes that $L_{\lambda_0^*}$ does not contain the origin.
\end{proof}

Then Theorem \ref{main1} follows by applying the following lemma finitely many times.
\begin{lem}\label{avoid-affine}
	Let $L\subset\R^d$ be a lattice and $V\subset L$ be a subgroup. Then for any $\tau_V\in L\backslash V$, there exists a lattice $L'\subset L$ such that $L'\cap (V+\tau_V)=\emptyset$.
\end{lem}
\begin{proof}
If $\dim(span_\Z\{V, \tau_V\})=\dim(V)$, find $u_{\dim(V)+1},\dots, u_d\in L$, if necessary, such that $$L':=span_\Z\{V, u_{\dim(V)+1},\dots, u_d \}$$ is a lattice. By our construction, $\tau_V\notin L'$. Hence $V+\tau_V\subset L'+\tau_V$, a coset of $L'$ in $L$ which does not intersect $L'$, as desired. 

If $\dim(span_\Z\{V, \tau_V\})>\dim(V)$, find $u_{\dim(V)+2},\dots, u_d\in L$, if necessary, such that $$L':=span_\Z\{V, 2\tau_V, u_{\dim(V)+2},\dots, u_d \}$$ is a lattice. Notice $L'\cap (V+\tau_V)$ is not empty if and only if there exist $v, v'\in V, \alpha\in\Z$ such that
$$v+2\alpha \tau_V = v'+ \tau_V. $$
If it happens, $(2\alpha-1)\tau_V\in V$, which contradicts the assumption that $\dim(span_\Z\{V, \tau_V\})>\dim(V)$.
\end{proof}

\section{Proof of Theorem \ref{refined}}\label{pf-refined}
We first study relations between $e$ and $\tau$ in planar zonotopes, equivalently symmetric convex polygons.
\begin{lem}\label{e=tau}
	Let $P$ be a zonotope as in Figure \ref{fig:polygon}. Then
	$$\tau_j-\tau_{j+1}=e_j+e_{j+1},\ \forall\ j=1,\dots,m-1.$$
	Furthermore, if $m$ is even, then for any $1\leq j\leq m$,
	$e_{j}$ is a linear combination of $\tau_{j'}, j'\neq j$, with coefficients $\pm 1$.
\end{lem}
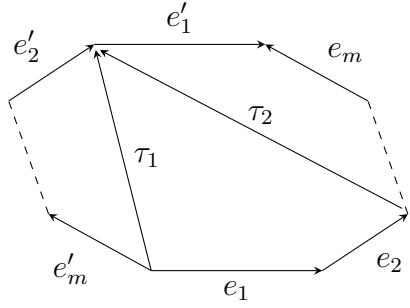
\begin{figure}[H]
\centering
\begin{tikzpicture}[scale=0.75, p2/.style={line width=0.275mm, black}, p3/.style={line width=0.15mm, black!50!white}]

\draw[-stealth] (-1, -2) -- (2, -2);
\draw[-stealth] (2,-2) -- (3.5,-1); 
\draw[-stealth] (2.8, 1) -- (1, 2);
\draw[dashed] (3.5,-1) -- (2.8,1);
\draw[-stealth] (-2, 2) -- (1,2);
\draw[-stealth] (3.4,-0.9)-- (-1.9, 1.9);
\draw[-stealth] (-3.5, 1) -- (-2,2); 
\draw[-stealth]  (-1, -2) -- (-2.8, -1);
\draw[dashed] (-2.8, -1) -- (-3.5,1);

\draw[-stealth] (-1,-2) -- (-1.975,1.9);

\draw (1.9, 1.5) node[anchor=south west]{$e_m$};
\draw (-1.9, -1.5) node[anchor=north east]{$e'_m$};
\draw (2.75, -1.5) node[anchor=north west]{$e_2$};
\draw (-2.75, 1.5) node[anchor=south east]{$e'_2$};
\draw (0.5, -2) node[anchor=north]{$e_1$};
\draw (-0.5, 2) node[anchor=south]{$e'_1$};
\draw (-1.5, 0) node[anchor=west]{$\tau_1$};
\draw (0.5, 0.75) node[anchor=west]{$\tau_2$};


\end{tikzpicture}
\caption{$\tau_1-\tau_2=e_1+e_2$}
\label{fig:difference-tau}
\end{figure}
\begin{proof}
	For convenience, denote $e_{j+m}=-e_j, 1\leq j\leq m$. It follows that
\begin{equation}
	\label{tau-e}
	\begin{aligned}
		\tau_1=&e_2+\cdots+e_{m}\\
		\tau_2=&e_3+\cdots+e_{m+1}\\
		\vdots&\\
		\tau_{m}=&e_{m+1}+\cdots+e_{2m-1}
	\end{aligned}.
\end{equation}
Then the differences between adjacent equalities imply \begin{equation}\label{difference-of-tau}\tau_j-\tau_{j+1}=e_j+e_{j+1},\ j=1,\dots,m-1\  (\text{see Figure \ref{fig:difference-tau}}). 
\end{equation}

If $m$ is even, by \eqref{tau-e}, \eqref{difference-of-tau},
\begin{equation}
\begin{aligned}\tau_1=&\left(e_2+e_3\right)+\cdots+\left(e_{m-2}+e_{m-1}\right) + e_{m}\\=& \left(\tau_2-\tau_3\right)+\cdots+\left(\tau_{m-2}-\tau_{m-1}\right)+e_{m},\end{aligned}\end{equation}
as desired. Similar argument works for any $e_j$, $1\leq j\leq m$.
\end{proof}

We also need a quantitative version of condition 2 in Bolle's theorem (Theorem \ref{Bol94}).
\begin{lem}\label{quantitative-Bol}
	Suppose  $L\subset\R^2$ is a lattice and $e\in L$, $\tau\in\R^2$. If there exists $t\in\R$ such  that $te+\tau\in L$, then $\det(e,\tau)$ is an integer multiple of $\det(L)$. Conversely, if $\det(\tau, e)\in \det(L) \Q$, there exists $t\in\R$ and a lattice $\widetilde{L}$ containing $L$ such that $te+\tau\in \widetilde{L}$.
\end{lem}
\begin{proof}
	We may assume $L=\Z^2$. If $te+\tau = u\in \Z^2$, then
	$$\det(e, \tau)= \det(e, te+\tau)= \det(e, u)\in\Z.$$
	Conversely, if $\det(\tau, e)\in\Q$, choose $t_0$ such that $(t_0e+\tau)\perp e$. Denote $e=(x_1, x_2)$ and $e^\perp=(-x_2, x_1)$. Say $(t_0e+\tau) = Ce^\perp$. Since $e\in\Z^2$ and $\det(\tau, e)\in \Q$, both $$\det(e, e^\perp),\ \det(e, \tau)=\det(e, t_0e+\tau)=C\det(e, e^\perp)$$ are rational. Hence $C\in\Q$, $t_0e+\tau\in\Q^2$ and $\widetilde{L}=span_\Z\{t_0e+\tau, \Z^2\}$ is the desired lattice.
\end{proof}

Now we can prove Theorem \ref{refined}. The ``if'' part follows from  Theorem \ref{Bol94} and the second half of Lemma \ref{quantitative-Bol}. The reason we do not need $e_{j_0}$ in the definition of $\Lambda_{j_0}$ is, when $m$ is even, $e_{j_0}$ is a linear combination of $\tau_j, j\neq j_0$, with coefficients $\pm 1$ (see Lemma \ref{e=tau}).

Conversely, assume $P$ multi-tiles. By Corollary \ref{SolProb1}, $P$ admits a lattice multi-tiling. Since the statement is invariant under non-degenerate linear transformations, we may assume $P$ is as in Figure \ref{fig:polygon} and $P+\Z^d$ multi-tiles. Denote
$$J=\left\{j\in\{1,\dots,m\}: e_j\notin\Z^d\right\}.  $$

\subsection*{Case 1} $\#(J)\geq 2$.

Say $J=\{j_1,\dots,j_{\#(J)}\}$. Denote $e^J_s=e_{j_s}, s=1,\dots,\#(J)$ and $P_J$ as the zonotope generated by $e^J_s, s=1,\dots,\#(J)$, namely
$$P_J=\left\{\sum_{s=1}^{\#(J)}\alpha_s e_s^J: \alpha_s\in [-\frac{1}{2}, \frac{1}{2}]\right\}.$$
Denote $\tau^J_s$ as the vector that translates $e^J_{s}$ to its parallel edge in $P_J$. 

Since $e_j, 1\leq j\leq m$ have increasing arguments, so do $e^J_s, 1\leq s\leq \#(J)$. Therefore Lemma \ref{e=tau} applies to $P_J$. We first show that $P_J+\Z^d$ is also a multiple tiling. To see this, observe that $P$ can be obtained by ``adding'' integer vectors $e_j, j\notin J$ into $P_J$. More precisely, apply \eqref{tau-e} to both $P$ and $P_J$, one can see that for each $s=1,\dots,\#(J)$, the difference between $\tau_s^J$ (in $P_J$) and $\tau_{j_s}$ (in $P$) is a linear combination of $e_j, j\notin J$ with coefficients $\pm 1$, which implies $\tau_s^J-\tau_{j_s}\in\Z^2$. Since each pair $(e_j, \tau_j), j=1,\dots, m$, satisfies conditions in Theorem \ref{Bol94} with respect to $\Z^2$, it follows that $(e^J_s, \tau^J_s), s=1,\dots,\#(J)$ also satisfy conditions in Theorem \ref{Bol94} with $\Z^2$. Hence $P_J+\Z^2$ is a multiple tiling.

\begin{claim}
	$\#(J)$ must be odd and there exists $\gamma\in\R^2$ such that 
	$$e^J_s\in\Z^d+(-1)^s\gamma, \ s=1,2,\dots, \#(J).$$
\end{claim}
Since $(e^J_s, \tau^J_s), s=1,\dots,\#(J)$ satisfy conditions in Theorem \ref{Bol94} with respect to $\Z^2$ but $e^J_s\notin\Z^2$, it follows that all $\tau^J_s\in \Z^2$. By Lemma \ref{e=tau}, $e^J_s+e^J_{s+1}\in\Z^2$ for any $1\leq s\leq \#(J)-1$, which implies there exists $\gamma\in\R^2$ such that
$$e^J_s\in\Z^d+(-1)^s\gamma, \ s=1,2,\dots, \#(J).$$

It remains to show $\#(J)$ must be odd. If $\#(J)$ is even, the second half of Lemma \ref{e=tau} implies all $e^J_{s}\in\Z^2$, contradiction. This completes the proof of Claim 1.

If $\gamma\in\Q^2$, then $e_j\in \Q^2$ for any $1\leq j\leq m$ and there is nothing to prove. So we may assume $\gamma\notin\Q^2$. Denote $J^c=\{1,\dots,m\}\backslash J$.

\begin{claim}
	$\#(J^c)=0,1$ if $\gamma\notin \Q^2$.
\end{claim}
Back to the original polygon $P$. Since $\#(J)$ is odd and 
$$e^J_s\in\Z^d+(-1)^s\gamma, \ s=1,2,\dots, \#(J),$$
 it follows that for any $j \notin J$, $\tau_j\in\Z^2\pm\gamma$ which does not lie in $\Z^2$. Thus condition $2$ in Theorem \ref{Bol94} must hold for $(e_j, \tau_j), \forall j\notin J$, with respect to $\Z^2$. If $\#(J^c)\geq 2$, there are $e_{j_0}, e_{j_0'}\in\Z^2\backslash\{0\}$, not parallel to each other, and $t, t'\in\R$ such that both
$$te_{j_0}+\gamma\in\Z^2,\  t'e_{j_0'}+\gamma\in\Z^2.  $$
See $e_{j_0}, e_{j_0'}$ as column vectors. Then
$$te_{j_0}-t'e_{j_0'}=(e_{j_0}, e_{j_0'})\begin{pmatrix}
	t\\-t'
\end{pmatrix}\in\Z^2.$$

Since $e_{j_0}, e'_{j_0'}\in\Z^2\backslash\{0\}$ are not parallel to each other, the matrix $(e_{j_0}, e_{j_0'})\in\Z_{2\times 2}$ is non-degenerate. Then both $t, t'$ are rational and $\gamma\in\Z^2-te_{j_0}\subset\Q^2$, contradiction. 
\vskip.125in
Now let us finish the case $\#(J)\geq 2$. 

If $\gamma\in\Q^2$, all $e_j$ are rational and conditions in Theorem \ref{refined} are satisfied.

If $\gamma\notin\Q$ and $J^c=\emptyset$, then $m=\#(J)$ must be odd. Also the definition of $J$ and Theorem \ref{Bol94} imply all $\tau_j\in\Z^2$. Hence $\Lambda_\tau$ is a sub-lattice of $\Z^2$, as desired. 

If $\gamma\notin\Q$ and $J^c=\{j_0\}$, then $m=\#(J)+1$ must be even. As we discussed right before Claim 1, $\tau_s^J\in\Z^2$ implies $\tau_{j_s}\in\Z^2$. Therefore $\tau_j\in\Z^2$ for any $j\neq j_0$ and $\Lambda_{j_0}$ is a sub-lattice of $\Z^2$. It remains to show $\det(e_{j_0},\tau_{j_0})\in\Q$. Since $\tau_{j_0}\notin\Z^2$ and $P+\Z^2$ multi-tiles, the pair $(e_{j_0}, \tau_{j_0})$ must satisfy condition $2$ in Theorem \ref{Bol94}, which, by Lemma \ref{quantitative-Bol}, implies $\det(e_{j_0},\tau_{j_0})\in\Z$, as desired.
\subsection*{Case 2}
	$\#(J)=1$.

	Say $J=\{j_0\}$. In this case $\#(J)$ is odd and there exists $\gamma\in\R^2$ such that $e_{j_0}\in\Z^2-\gamma$. If $\gamma\in \Q$, all $e_j$ are rational and there is nothing to prove. If $\gamma\notin\Q^2$, the proof of Claim 2 in Case 1 still works and $P$ turns out to be a parallelogram.

\subsection*{Case 3}
$\#(J)=0$. Trivial.
\vskip.125in
Above all, we proved that if $P$ is not a parallelogram and $P+\Z^2$ multi-tiles, then $P$ must satisfy one of the following.
\begin{enumerate}
	\item $e_j\in\Q^2$, for any $1\leq j\leq m$. ($\#(J)=0, 1$, or $\#(J)\geq 2, \gamma\in\Q^2$)

	\item $m$ is odd and $\tau_j\in\Z^2$ for any $1\leq j\leq m$. ($\#(J)\geq 2, \gamma\notin\Q$, $\#(J^c)=0$)

	\item $m$ is even, there is a unique $j_0\in\{1,\dots,m\}$ such that $\tau_j\in\Z^2$ for any $j\neq j_0$,  $e_{j_0}\in\Z^2$ and $\det(e_{j_0}, \tau_{j_0})$ is a rational multiple of $\det(\Lambda_{j_0})$. ($\#(J)\geq 2, \gamma\notin\Q$, $\#(J^c)=1$)
\end{enumerate}
Any case above satisfies condition $1$ or $2$ in Theorem \ref{refined}. Also one can see $L_P$ is a discrete subgroup of additive group $\Q^2$. If one can show $L_P$ is full-rank, it is not hard to check that $L_P\cap\Z^2$ is also full-rank, which completes the proof.

Now it remains to show $L_P$ is full-rank. When $m$ is odd, $L_P=\Lambda_\tau$ must be full-rank so there is nothing to prove. When $m$ is even, we shall show that if there exists another $j_0'$ such that $\Lambda_{j_0'}$ is also a lattice, then $\tau_j\in\Q^2$ for any $1\leq j\leq m$. It is already proved above that $\Lambda_{j_0}$ is a lattice in $\Q^2$, so it remains to show $\tau_{j_0}\in\Q^2$. Since $P$ is not a parallelogram and $m\geq 4$, $\{\tau_j,\, j\neq j_0, j_0'\}$ generate a sub-lattice of $\Lambda_{j_0}\subset \Q^2$. Since $\Lambda_{j'_0}$ is also a lattice, $\tau_{j_0}$ is rationally dependent with $\{\tau_j,\, j\neq j_0, j_0'\}$, thus must be rational, as desired. In fact in this case $e_j\in\Q^2$ for any $1\leq j\leq m$ (see Lemma \ref{e=tau}).

\section{Proof of Theorem \ref{periodic-multi-tilings}}\label{pf-perodic}
Theorem \ref{periodic-multi-tilings} follows from Theorem \ref{Kol00}, Theorem \ref{main1}, Theorem \ref{refined} and the following lemma.
\begin{lem}\label{sub-lattice}
	Let $L\subset\R^d$ be a lattice and $L_1,\dots,L_n\subset L$ are sub-lattices. Then 
	$$\bigcap_{j=1}^n L_j  $$
	is a sub-lattice of $L$.
\end{lem}
\begin{proof}
	It suffices to prove the case $n=2$. If $L_1 \cap L_2$ is not full-rank, there exists $u\in L_1$ such that $$\dim(span_\Z\{L_1 \cap L_2, u\})>\dim(L_1 \cap L_2).  $$
Since $u\in L_1$, $\Z u \cap L_2$ must be trivial, which implies
$$\dim(span_\Z\{L_2, u\})>\dim(L_2)=d,  $$
contradiction.
\end{proof}

Now we can complete the proof. By Theorem \ref{Kol00}, if $P\subset\R^2$ is not a parallelogram and $P+\Lambda$ multi-tiles, $\Lambda$ must be a finite union of translated lattices, that is,
$$\delta_\Lambda = \sum_{j=1}^n \delta_{L_j+{z}_j}. $$
By Theorem \ref{main1}, for each $j$ there exists a lattice $\widetilde{L_j}$ containing $L_j$ such that $P+\widetilde{L_j}$ multi-tiles. By Theorem \ref{refined}, $\widetilde{L_j}\cap L_P$ is a lattice, where $L_P$ is defined in \eqref{L_P}. Since both $L_j$ and $\widetilde{L_j}\cap L_P$ are sub-lattices of $\widetilde{L_j}$, by Lemma \ref{sub-lattice} $L_j\cap L_P$ is a lattice. Therefore, by Lemma \ref{sub-lattice} again, 
$$ \bigcap_{j=1}^n \left(L_j\cap L_P\right)=\left(\bigcap_{j=1}^n L_j\right)\cap L_P, $$
which is a finite union of sub-lattices of $L_P$, is a lattice. Hence $\bigcap L_j$ is full-rank and $\Lambda$ is a finite union of translations of $\bigcap L_j$.

\section{Examples}\label{example}
\subsection{Symmetric polygons that do not multi-tile by translations}\label{do-not-multi-tile}
We shall show that for any $m\geq 4$, there exist symmetric $(2m)$-gons that do not multi-tile by translations. As far as the author knows, these are the first known symmetric polygons that do not multi-tile by translations. Since there are many symmetric $(2m)$-gons that do multi-tile by translations, this means, unlike tiling (Venkov-McMullen), one can not determine whether a polygon multi-tiles only by its combinatorial type.

Take a zonotope $P$ as in Figure \ref{fig:polygon} such that $e_j, 1\leq j\leq m$ are rationally independent. 

When $m$ is odd, if $P$ multi-tiles, by Theorem \ref{refined} all $\tau_j$ generate a lattice. Since $m-1\geq 3$, by \eqref{difference-of-tau}, $\tau_j-\tau_{j+1}=e_j+e_{j+1}, 1\leq j\leq m-1$, are rationally dependent, namely there exists $q_1,\dots,q_{m-1}\in\Q$, not all $0$, such that
\begin{equation*}
\begin{aligned}
0= &q_1(e_1+e_2)+\cdots+q_{m-1}(e_{m-1}+e_{m})\\=&q_1e_1+(q_1+q_2)e_2+\cdots +(q_{m-2}+q_{m-1})e_{m-1}+q_{m-1}e_m. 
\end{aligned}
\end{equation*}
It follows that $q_1=0, q_1+q_2=0,\dots, q_{m-2}+q_{m-1}=0, q_{m-1}=0$, which implies $q_j=0$ for any $1\leq j\leq m-1$, contradiction.

When $m$ is even, we may assume $j_0$ in Theorem \ref{refined} equals $1$. Since $e_1$ is a linear combination of $\tau_j, j\geq 2$, with coefficients $\pm 1$ (see Lemma \eqref{e=tau}), $e_1, \tau_2,\dots,\tau_m$ generate a lattice. Since $m-1\geq 3$, by \eqref{difference-of-tau}, $e_1$ and $\tau_j-\tau_{j+1}=e_j+e_{j+1}, 2\leq j\leq m-1$ are rationally dependent, namely there exists $q_1,\dots,q_{m-1}\in\Q$, not all $0$, such that
\begin{equation*}
\begin{aligned}
0= &q_1e_1+q_2(e_2+e_3)+\cdots+q_{m-1}(e_{m-1}+e_{m})\\=&q_1e_1+q_2e_2+(q_2+q_3)e_3+\cdots +(q_{m-2}+q_{m-1})e_{m-1}+q_{m-1}e_m. 
\end{aligned}
\end{equation*}
It follows that $q_1=0, q_2=0,q_2+q_3=0\dots, q_{m-2}+q_{m-1}=0, q_{m-1}=0$, which implies $q_j=0$ for any $1\leq j\leq m-1$, contradiction.

In fact, the summary at the end of Section \ref{pf-refined} says, if $P+\Z^2$ multi-tiles, then $e_j\in(\Z^2\pm\gamma)\cup \Q^2$, for some $\gamma\in\R^2$. Therefore we only need a quadruple of rationally independent $e_j$ to deny multiple tilings of $P$. We omit the proof.

\subsection{Some non-periodic multi-tilings}\label{non-periodic}
We shall show that non-parallelogram and convexity are necessary in Theorem \ref{periodic-multi-tilings}. In fact we shall construct multiple tilings where $\Lambda+\alpha\neq\Lambda$ for any $\alpha\in\R^2\backslash\{0\}$.

It is very easy for parallelograms. One can simply take $P=[0,1]^2$,
$$\Lambda_1=\left(\Z\times(\Z\backslash\{0\})\right)\cup \left(\Z+\beta\right)\times\{0\}, \ \ \Lambda_2=\left((\Z\backslash\{0\})\times\Z\right)\cup \{0\}\times \left(\Z+\beta'\right),$$
for $\beta, \beta'\notin\Z$ and take $\Lambda=\Lambda_1\cup\Lambda_2$

For the convexity, one example is the skew tetromino (see Figure \ref{fig:convexity-necessary}), which is a union of $4$ unit squares.

\begin{figure}[H]
\centering
\begin{subfigure}{.5\textwidth}
\centering
\begin{tikzpicture}[scale=0.75, p2/.style={line width=0.275mm, black}, p3/.style={line width=0.15mm, black!50!white}]

\draw (-1,-1) -- (0, -1) -- (0,0) -- (1,0) -- (1,2) -- (0,2) -- (0,1) -- (-1,1) -- (-1,-1) ;
\draw[dashed] (-1,0) -- (0,0) -- (0,1) -- (1,1);
\draw (0, -1) node[anchor=north]{P};

\end{tikzpicture}
\end{subfigure}
\begin{subfigure}{.5\textwidth}
\centering
\begin{tikzpicture}[scale=0.5, p2/.style={line width=0.275mm, black}, p3/.style={line width=0.15mm, black!50!white}]

\draw[thin, pattern=north east lines, shift={(0,1)}] (-1,-1) -- (0, -1) -- (0,0) -- (1,0) -- (1,2) -- (0,2) -- (0,1) -- (-1,1) -- (-1,-1) ;
\draw[thin, pattern=north east lines,shift={(0,-1)}] (-1,-1) -- (0, -1) -- (0,0) -- (1,0) -- (1,2) -- (0,2) -- (0,1) -- (-1,1) -- (-1,-1) ;
\draw[thin, pattern=north east lines,shift={(0,3)}] (-1,-1) -- (0, -1) -- (0,0) -- (1,0) -- (1,2) -- (0,2) -- (0,1) -- (-1,1) -- (-1,-1) ;
\draw[thin, pattern=north east lines,shift={(0,5)}] (-1,-1) -- (0, -1) -- (0,0) -- (1,0) -- (1,2) -- (0,2) -- (0,1) -- (-1,1) -- (-1,-1) ;
\draw[shift={(2,2)}] (-1,-1) -- (0, -1) -- (0,0) -- (1,0) -- (1,2) -- (0,2) -- (0,1) -- (-1,1) -- (-1,-1) ;
\draw[shift={(2,4)}] (-1,-1) -- (0, -1) -- (0,0) -- (1,0) -- (1,2) -- (0,2) -- (0,1) -- (-1,1) -- (-1,-1) ;
\draw[shift={(-2,2)}] (-1,-1) -- (0, -1) -- (0,0) -- (1,0) -- (1,2) -- (0,2) -- (0,1) -- (-1,1) -- (-1,-1) ;
\draw[shift={(-2,4)}] (-1,-1) -- (0, -1) -- (0,0) -- (1,0) -- (1,2) -- (0,2) -- (0,1) -- (-1,1) -- (-1,-1) ;
\draw[shift={(2,0)}] (-1,-1) -- (0, -1) -- (0,0) -- (1,0) -- (1,2) -- (0,2) -- (0,1) -- (-1,1) -- (-1,-1) ;
\draw[shift={(2,-2)}] (-1,-1) -- (0, -1) -- (0,0) -- (1,0) -- (1,2) -- (0,2) -- (0,1) -- (-1,1) -- (-1,-1) ;
\draw[shift={(-2,0)}] (-1,-1) -- (0, -1) -- (0,0) -- (1,0) -- (1,2) -- (0,2) -- (0,1) -- (-1,1) -- (-1,-1) ;
\draw[shift={(-2,-2)}] (-1,-1) -- (0, -1) -- (0,0) -- (1,0) -- (1,2) -- (0,2) -- (0,1) -- (-1,1) -- (-1,-1) ;
\draw[shift={(-4,-2)}] (-1,-1) -- (0, -1) -- (0,0) -- (1,0) -- (1,2) -- (0,2) -- (0,1) -- (-1,1) -- (-1,-1) ;
\draw[shift={(-4,0)}] (-1,-1) -- (0, -1) -- (0,0) -- (1,0) -- (1,2) -- (0,2) -- (0,1) -- (-1,1) -- (-1,-1) ;
\draw[shift={(-4,2)}] (-1,-1) -- (0, -1) -- (0,0) -- (1,0) -- (1,2) -- (0,2) -- (0,1) -- (-1,1) -- (-1,-1) ;
\draw[shift={(-4,4)}] (-1,-1) -- (0, -1) -- (0,0) -- (1,0) -- (1,2) -- (0,2) -- (0,1) -- (-1,1) -- (-1,-1) ;
\draw[shift={(4,-2)}] (-1,-1) -- (0, -1) -- (0,0) -- (1,0) -- (1,2) -- (0,2) -- (0,1) -- (-1,1) -- (-1,-1) ;
\draw[shift={(4,0)}] (-1,-1) -- (0, -1) -- (0,0) -- (1,0) -- (1,2) -- (0,2) -- (0,1) -- (-1,1) -- (-1,-1) ;
\draw[shift={(4,2)}] (-1,-1) -- (0, -1) -- (0,0) -- (1,0) -- (1,2) -- (0,2) -- (0,1) -- (-1,1) -- (-1,-1) ;
\draw[shift={(4,4)}] (-1,-1) -- (0, -1) -- (0,0) -- (1,0) -- (1,2) -- (0,2) -- (0,1) -- (-1,1) -- (-1,-1) ;

\draw (0, -3) node[anchor=north]{$P+\Lambda_1$};

\end{tikzpicture}
\end{subfigure}%
\begin{subfigure}{.5\textwidth}
\centering
\begin{tikzpicture}[scale=0.5, p2/.style={line width=0.275mm, black}, p3/.style={line width=0.15mm, black!50!white}]
\draw[shift={(-1,-1)}] (-1,-1) -- (0, -1) -- (0,0) -- (1,0) -- (1,2) -- (0,2) -- (0,1) -- (-1,1) -- (-1,-1) ;
\draw[thin, pattern=north east lines, shift={(-4,6)}] (-1,-1) -- (0, -1) -- (0,0) -- (1,0) -- (1,2) -- (0,2) -- (0,1) -- (-1,1) -- (-1,-1) ;
\draw[shift={(-1,1)}] (-1,-1) -- (0, -1) -- (0,0) -- (1,0) -- (1,2) -- (0,2) -- (0,1) -- (-1,1) -- (-1,-1) ;
\draw[thin, pattern=north east lines, shift={(0,4)}] (-1,-1) -- (0, -1) -- (0,0) -- (1,0) -- (1,2) -- (0,2) -- (0,1) -- (-1,1) -- (-1,-1) ;
\draw[shift={(1,1)}] (-1,-1) -- (0, -1) -- (0,0) -- (1,0) -- (1,2) -- (0,2) -- (0,1) -- (-1,1) -- (-1,-1) ;
\draw[shift={(1,3)}] (-1,-1) -- (0, -1) -- (0,0) -- (1,0) -- (1,2) -- (0,2) -- (0,1) -- (-1,1) -- (-1,-1) ;
\draw[thin, pattern=north east lines, shift={(-2,2)}] (-1,-1) -- (0, -1) -- (0,0) -- (1,0) -- (1,2) -- (0,2) -- (0,1) -- (-1,1) -- (-1,-1) ;
\draw[thin, pattern=north east lines, shift={(-2,4)}] (-1,-1) -- (0, -1) -- (0,0) -- (1,0) -- (1,2) -- (0,2) -- (0,1) -- (-1,1) -- (-1,-1) ;
\draw[shift={(1,-1)}] (-1,-1) -- (0, -1) -- (0,0) -- (1,0) -- (1,2) -- (0,2) -- (0,1) -- (-1,1) -- (-1,-1) ;
\draw[thin, pattern=north east lines, shift={(-2,6)}] (-1,-1) -- (0, -1) -- (0,0) -- (1,0) -- (1,2) -- (0,2) -- (0,1) -- (-1,1) -- (-1,-1) ;
\draw[shift={(-3,-1)}] (-1,-1) -- (0, -1) -- (0,0) -- (1,0) -- (1,2) -- (0,2) -- (0,1) -- (-1,1) -- (-1,-1) ;
\draw[thin, pattern=north east lines, shift={(0,6)}] (-1,-1) -- (0, -1) -- (0,0) -- (1,0) -- (1,2) -- (0,2) -- (0,1) -- (-1,1) -- (-1,-1) ;
\draw[thin, pattern=north east lines, shift={(2,6)}] (-1,-1) -- (0, -1) -- (0,0) -- (1,0) -- (1,2) -- (0,2) -- (0,1) -- (-1,1) -- (-1,-1) ;
\draw[thin, pattern=north east lines, shift={(-4,0)}] (-1,-1) -- (0, -1) -- (0,0) -- (1,0) -- (1,2) -- (0,2) -- (0,1) -- (-1,1) -- (-1,-1) ;
\draw[thin, pattern=north east lines, shift={(-4,2)}] (-1,-1) -- (0, -1) -- (0,0) -- (1,0) -- (1,2) -- (0,2) -- (0,1) -- (-1,1) -- (-1,-1) ;
\draw[thin, pattern=north east lines, shift={(-4,4)}] (-1,-1) -- (0, -1) -- (0,0) -- (1,0) -- (1,2) -- (0,2) -- (0,1) -- (-1,1) -- (-1,-1) ;
\draw[shift={(3,-1)}] (-1,-1) -- (0, -1) -- (0,0) -- (1,0) -- (1,2) -- (0,2) -- (0,1) -- (-1,1) -- (-1,-1) ;
\draw[shift={(3,1)}] (-1,-1) -- (0, -1) -- (0,0) -- (1,0) -- (1,2) -- (0,2) -- (0,1) -- (-1,1) -- (-1,-1) ;
\draw[shift={(3,3)}] (-1,-1) -- (0, -1) -- (0,0) -- (1,0) -- (1,2) -- (0,2) -- (0,1) -- (-1,1) -- (-1,-1) ;
\draw[shift={(3,5)}] (-1,-1) -- (0, -1) -- (0,0) -- (1,0) -- (1,2) -- (0,2) -- (0,1) -- (-1,1) -- (-1,-1) ;

\draw (0, -2) node[anchor=north]{$P+\Lambda_2$};

\end{tikzpicture}
\end{subfigure}
\caption{The skew tetromino and two tiles}
\label{fig:convexity-necessary}
\end{figure}
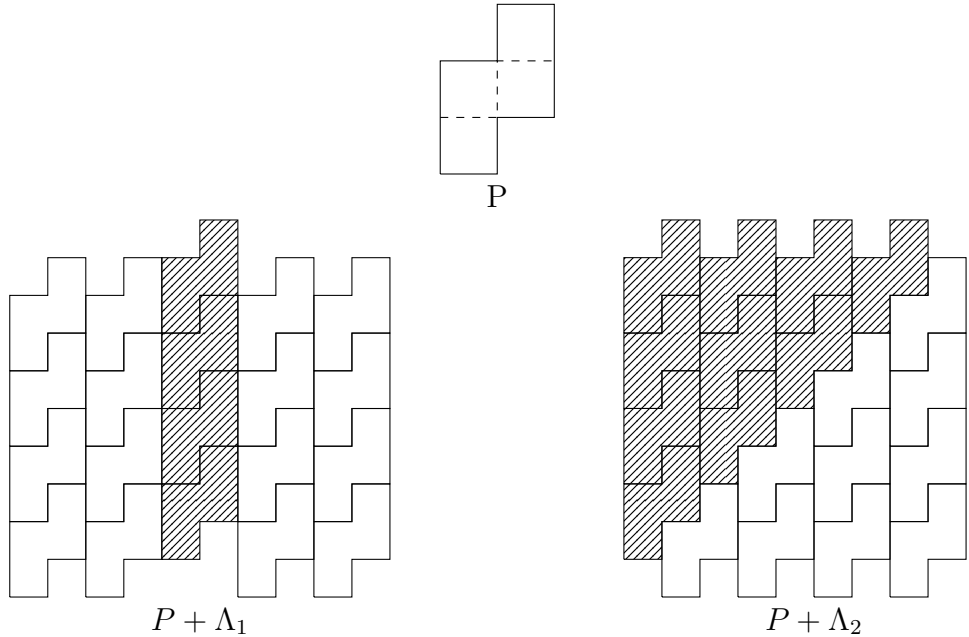
As shown in Figure \ref{fig:convexity-necessary}, both $P+\Lambda_1$, $P+\Lambda_2$ tile, where 
$$\Lambda_1= ((2\Z\backslash\{0\})\times 2\Z)\cup(\{0\}\times(2\Z+1)),  $$
$$\Lambda_2=\{(m,n)\in (2\Z)^2: m\geq n\}\cup \{(m,n)\in (2\Z-1)^2: m<n\}.$$
Then $P+(\Lambda_1\cup\Lambda_2)$ is a multiple tiling where $\Lambda+\alpha\neq\Lambda$ for any $\alpha\in\R^2\backslash\{0\}$.

\subsection{A family of indecomposable multi-tilings by a symmetric non-regular octagon}\label{non-rationality}
We shall construct a family of indecomposable multiple tilings of an octagon, where each discrete set is
 $$\Lambda=\Z\times 2\Z +\{\vec{0}, \alpha\},\ \text{for some}\ \alpha\notin\Q^2.$$
This means, even with trivial counterexamples ruled out (parallelograms, decomposable multi-tilings), the analog of Lagarias-Wang's rationality theorem on $1$-dimensional tiling still fails for multiple tilings in the plane. See the discussion after Theorem \ref{periodic-multi-tilings}.

Let $P$ be the symmetric octagon in Figure \ref{fig:octagon} below whose vertices lie in $\Z^2$.
\begin{figure}[H]
\centering
\begin{tikzpicture}[scale=1, p2/.style={line width=0.275mm, black}, p3/.style={line width=0.15mm, black!50!white}]
\draw[->] (0,0) -- (4, 0);
\draw[->] (0,0) -- (0,4);
\foreach \x/\xtext in {1,2,3}
\draw (\x, 0) node[below]{$\xtext$};
\foreach \y/\ytext in {1,2,3}
\draw (0, \y) node[left]{$\ytext$};
\draw (0,0) node[anchor=north east]{$0$};
\draw[very thin,color=gray] (0,0) grid (3,3);
\draw (1,0) -- (2,0) -- (3,1) -- (3,2) -- (2,3) -- (1,3) -- (0,2) -- (0,1) -- (1,0);
\end{tikzpicture}
\caption{A symmetric octagon whose vertices lie in $\Z^2$}
\label{fig:octagon}
\end{figure}
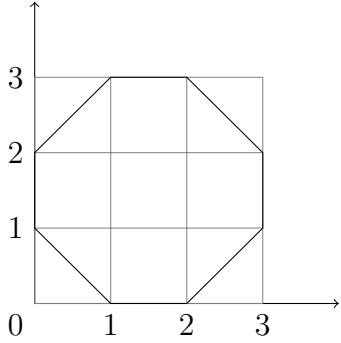
One can check that on each horizontal strip $\R\times[n, n+1]$,
$$P+\Z\times 2\Z=\begin{cases}
	4, & n \text{ is even}\\3, &n \text{ is odd}
\end{cases}.  $$
Therefore for any $\beta\in\R$, 
$$P+\left(\Z\times 2\Z+\{(0,0), (\beta, 1)\}\right)  $$
is a multi-tiling of multiplicity $7$. We claim this multi-tiling is indecomposable. To see this, clearly it does not tile and it is proved by Yang and Zong (\cite{YZ17-lattice}, \cite{YZ17}) that except parallelograms and hexagons, no convex polygon admits translational multi-tilings of multiplicities $2,3,4$. This means a $7$-tiling by an octagon can not be decomposed into two multiple tilings.

\
\begin{appendix}
\section*{Appendix: Convex bodies that multi-tile by translations must be convex polytopes}
\begin{thm*}
Suppose $P\subset\R^d$ is a convex body and there exists a discrete multi-set $\Lambda$ such that $P+\Lambda$ multi-tiles. Then $P$ is a convex polytope. 
\end{thm*}
\begin{proof}
Notice $\partial P+\Lambda$ decomposes $\R^d$ into disjoint (open) cells. We first show each cell is convex. Pick a cell $C$, for any $\lambda\in\Lambda$, $$C\subset P+\lambda,\  or\  C\cap (P+\lambda)=\emptyset.$$ Say the multiplicity of $P+\Lambda$ is $k$. Then there exists $\lambda_1,\dots, \lambda_k$ such that
$$C\subset\bigcap_{j=1}^k (P+\lambda_j).  $$
We claim they are actually equal. If not, there exists another cell $C'$ such that
$$C'\subset\bigcap_{j=1}^k (P+\lambda_j).$$
Since $C$ and $C'$ are two different cells, they can be separated by $\partial P+\lambda'$ for some $\lambda'\in\Lambda$, that is
$$C\subset(P+\lambda'),\ C'\cap(P+\lambda')=\emptyset,  $$
or
$$C'\subset(P+\lambda'),\ C\cap(P+\lambda')=\emptyset.  $$
In either case, $\lambda'$ is not equal to any of $\lambda_1,\dots,\lambda_k$, which means $C$ or $C'$ is covered at least $k+1$ times. Contradiction.

Next, fix a convex cell $C_0$, for any other convex cell $C$, there exists a half-space $H_C$ such that
$$C_0\subset H_C, \ C\cap H_C=\emptyset,$$
which implies
$$C\subset \bigcap_{C\neq C_0} H_C. $$
Since all cells tile $\R^d$, it follows that up to measure $0$
$$C_0=\bigcap_{C\neq C_0} H_C.$$

Now it suffices to show $C_0$ is in fact an intersection of finitely many half-spaces. Since $diam(C)$ is bounded above uniformly, we may assume $dist (\partial H_C, C_0)$ is large when $dist(C, C_0)$ is large. Choose finitely may $H_C$ whose intersection is bounded. Then when $dist(C, C_0)$ is large, $\partial H_C$ is far from $C_0$ and therefore dropping $H_C$ does not change the intersection, as desired. Thus $C$ must be a polytope. Since the original convex body $P$ is a union of finitely many cells, it must be a convex polytope.
\end{proof}
\end{appendix}
\
\bibliographystyle{abbrv}
\bibliography{/Users/MacPro/Dropbox/Academic/paper/mybibtex.bib}

\end{document}